\documentclass [twoside,reqno] {amsart}

\usepackage{amsfonts}
\usepackage{amssymb, tikz}
\usepackage{color}
\usepackage{hyperref}

\newtheorem{thm}{Theorem}

\newtheorem{lem}[thm]{Lemma}
\newtheorem{prop}[thm]{Proposition}

\newtheorem{defn}[thm]{Definition}
\newtheorem{rem}[thm]{Remark}

\theoremstyle{definition}

\newcommand{\cD}{\mathcal{D}}
\newcommand{\bD}{\mathbb{D}^{n}}
\newcommand{\bS}{\mathbb{S}^{n-1}}

\newcommand{\ghd}[1]{d_{GH}^{#1}}
\newcommand{\hd}[1]{d_{H}^{#1}}

\newcommand{\R}{\mathbb{R}}
\newcommand{\cR}{\mathcal{R}}
\newcommand{\cB}{\mathcal{B}_T}

\newcommand{\eps}{\varepsilon}
\newcommand{\abs}[1]{\left\lvert#1\right\rvert}
\newcommand{\inwb}[1]{\partial_{\mathrm{in}}SM_{#1}}
\newcommand{\inwd}[1]{\partial_{\mathrm{in}}S\bD_{#1}}

\newcommand{\outwb}{\partial_{\mathrm{out}}SM}

\newcommand{\pM}{{\partial M}}

\newcommand{\exit}{\tau_{\mathrm{exit}}}

\def \bfo {\begin {eqnarray*} }
\def \efo {\end {eqnarray*} }
\def \ba {\begin {eqnarray*} }
\def \ea {\end {eqnarray*} }
\def \beq {\begin {eqnarray}}
\def \eeq {\end {eqnarray}}

\def \p {\partial}

\def \bfo {\begin {eqnarray*} }
\def \efo {\end {eqnarray*} }
\def \ba {\begin {eqnarray*} }
\def \ea {\end {eqnarray*} }
\def \beq {\begin {eqnarray}}
\def \eeq {\end {eqnarray}}

\def \p {\partial}

\begin{document}

	\title[Travel time inverse problems on simple Riemannian manifolds]{Three travel time inverse problems on simple Riemannian manifolds}
	
	\author[J. Ilmavirta]{Joonas Ilmavirta}
\address{Department of Mathematics and Statistics\\University of Jyv\"askyl\"a, Jyv\"askyl\"a,  FI-40014,  Finland   (\tt{joonas.ilmavirta@jyu.fi})}
   
  \author[B. Liu]{Boya Liu}
\address{Department of Mathematics,  North Carolina State University, Raleigh, NC 27695, USA   (\tt{bliu35@ncsu.edu})}
   
  \author[T. Saksala]{Teemu Saksala}
\address{Department of Mathematics,  North Carolina State University, Raleigh, NC 27695, USA   (\tt{tssaksal@ncsu.edu})}
   


%
	\date{\today}

	\begin{abstract}
	We provide new proofs based on the Myers--Steenrod theorem to confirm that travel time data, travel time difference data and the broken scattering relations determine a simple Riemannian metric on a disc up to the natural gauge of a boundary fixing diffeomorphism. Our method of the proof leads to a Lipschitz-type stability estimate for the first two data sets in the class of simple metrics.      
	\end{abstract}
	
		\maketitle
	
	\section{Introduction}
	
We study three geometric inverse problems.
The task is to reconstruct a simple Riemannian manifold from (1) travel time data, (2) travel time difference data, or (3) the broken scattering relations.
A compact Riemannian manifold is called simple if the boundary is strictly convex and the exponential map is a global diffeomorphism.

The third problem has been solved previously for dimension three and higher in general geometry, and we extend the result to dimension 2 in simple geometry.
For the other problems previous results give H\"older or similar stability in great generality, and we show that in simple geometry we improve this to Lipschitz.
In many previously known cases (e.g. uniqueness on simple manifolds for problems~1 and~2) our proofs are elementary due to simple geometry.

If the Earth is modeled as a Riemannian manifold (which is a popular albeit not exactly accurate choice; cf.~\cite{dehoop2020determination, de2020foliated}), these problems correspond simplified versions of problems arising from seismology.
The Earth is full of small-scale inhomogeneities that cause seismic waves to scatter. Seen on the scale of the whole planet, these inhomogeneities such as rocks are best seen as point-like objects of zero size even though a single point is irrelevant from the point of view of the elastic wave equation.
The model for scattering is mesoscopic: 
the scatterers are big enough to cause the waves to scatter but small enough to be considered mere points on the global scale that fill the planet densely.
For the first two problems no such considerations are needed, and the point sources can be either natural (e.g. earthquakes) or artificial (e.g. produced by focusing waves).
Seismic ray paths correspond to geodesics.
We will focus on the geometric problems in this paper and ignore details of the physical model.
We introduce the inverse problems in physical terms below and in mathematical terms in the following subsections.
	
	
In the first and second inverse problems we record the arrival times of seismic waves, initiated by a dense set of point sources going off at some unknown origin times, measured on a dense set of receivers on the surface of the planet.
The arrival time of a wave is the difference of the \emph{travel time} (distance from the source to receiver) and the origin time.
In our first main result (Theorem~\ref{thm:stability}), we assume that origin times are known and show that the collection of travel times stably determine the isometry class of a simple Riemannian metric.
If the origin time of a seismic event is unknown and we have measured the corresponding arrival times, then by computing the difference of arrival times we obtain a \emph{travel time difference}.
In our second main result (Theorem~\ref{thm:stability_of_TTDD}), we show that the travel time differences stably determine a simple metric up to a Riemannian isometry.  
    
In our third inverse problem, the data consists of the knowledge of ingoing directions on the boundary of the planet and for each such direction the set of all outcoming directions and the associated travel times of all single-scattered geodesics. This data is encoded in the \emph{broken scattering relations}. Our third main result (Theorem~\ref{thm:BSR_to_geometry}) shows that this data determines a simple metric up to a Riemannian isometry.

	
	\subsection{Inverse problem 1: Travel time data}
	
	Let $(M,g)$ be a compact, connected, oriented, and smooth Riemannian manifold with smooth boundary $\p M$. For any points $x, y \in M$, we use the notation $d(x, y)$ for the Riemannian distance between them. The first data set we study in this paper is the \emph{travel time data}.
	
	\begin{defn}
	\label{defn:TTD}
	For every point $p\in M$ its \emph{travel time function} $r_p:\p M \to \R$ is defined by the formula $r_p(z) = d(p, z)$. \emph{The travel time map} of the Riemannian manifold $(M,g)$ is then given by the formula
	\begin{equation}
		\label{eq:map_R}
		\cR:(M,g) \to (C(\p M),\|\cdot\|_\infty)
	\end{equation}
	with $\cR(p) = r_p$.
	The image set 
	$
	\cR(M)\subset C(\p M)
	$
	of the travel time map is called the travel time data of the Riemannian manifold $(M,g)$. 
	\end{defn}
	
	\begin{rem}
	The travel time data of a Riemannian manifold is an unlabelled collection of travel time functions. The locations of point sources related to these functions are unknown.
	\end{rem}
	
	It was showed in \cite[Section 3.8]{Katchalov2001} and~\cite{kurylev1997multidimensional}  that the travel time data determines any compact, connected, oriented, and smooth Riemannian manifold up to an isometry. A corresponding uniqueness result for a partial travel time data of a compact manifold with strictly convex boundary was established in~\cite{pavlechko2022uniqueness}. If the metric is Finslerian, 
	then the travel time data can be used to determine the metric, up to a natural obstruction, in the directions of the tangent bundle corresponding with distance minimizing geodesics that reach the boundary~\cite{dehoop2020determination}. The authors also verified that this is the largest set in which a Finsler metric can be determined from its travel time data. Approximate reconstruction of the Riemannian manifold and H\"older-stability of the travel time data, under certain geometric bounds, have been studied in  \cite{de2021stable, katsuda2007stability}. 
	
	In the present paper we study the uniqueness and Lipschitz-stability of the travel time data on simple Riemannian manifolds.
	It is important for our proofs that on a simple manifold there are no conjugate points, and consequently any two points on the manifold, including the boundary, are joined by a unique geodesic depending smoothly on the endpoints.
	
	We will make use of the Myers--Steenrod theorem: ``Every distance-preserving map between two connected Riemannian manifolds is actually a smooth isometry of Riemannian manifolds  \cite{myers-steenrod,palais-myers-steenrod}.''
	This result is true also on simple manifolds with boundary \cite[Lemma~29]{de2021stable}.
	In Proposition~\ref{prop:isometry} of this paper we use this theorem to show that the travel time map on simple Riemannian manifolds is a metric isometry and provide a new, short proof for Proposition~\ref{prop:TTD_to_geom} for the aforementioned uniqueness result in the class of simple Riemannian manifolds.
	Our first main result (Theorem~\ref{thm:stability}) is a Lipschitz-type stability result for the travel time data on simple Riemannian manifolds.

	\subsection{Inverse problem 2: Travel time difference data}
	
	The second data set studied in this paper is the \emph{travel time difference data}. 
	
	\begin{defn}
    \label{def:travel_time_diff}
    The travel time difference function of a point $p \in M$ is the function
    \[
    \begin{split}
    &D_p\colon \p M \times \p M \to \R,
    \\&
    D_p(z,w)=d(p,z)-d(p,w).
    \end{split}
    \]
    Then the \emph{travel time difference map} and the travel time difference data of the Riemannian manifold $(M,g)$ are
    \[
    \cD\colon (M,g) \to (C(\p M\times \p M),\|\cdot\|_\infty)
    \]
    with $\cD(p)=\frac{1}{2}D_p$,
    and its image set
    \[
    \cD(M)\subset C(\p M \times \p M),
    \]
    respectively.
    \end{defn}
    
    The factor $\frac12$ in $\cD(p)=\frac{1}{2}D_p$ is convenient in making the map $\cD$ an isometry.

    This data was introduced in~\cite{LaSa}, where instead of doing boundary measurements, the travel time difference $D_p(z,w)$ was measured for every source point $p \in N$ between any receiver points $z,w \in F$, where $F\subset N$ contains an open subset of a closed Riemannian manifold $(N,g)$. The authors showed that the knowledge of the metric on $F$, in conjunction with the travel time difference data, determine the isometry type of $(N,g)$. This result was later extended for the boundary measurements on compact manifolds in~\cite{de2018inverse}. On complete manifolds the uniqueness and a non-quantitative stability of the travel time difference data was obtained in \cite{ivanov2018distance, ivanov2022distance}. 
    
    In the present paper we revisit the travel time difference data on simple Riemannian manifolds. We provide a new short proof for the uniqueness of the geometry from this data (Proposition~\ref{prop:TTDD_to_geom}) by utilizing the Myers--Steenrod Theorem. Our second main result is a Lipschitz-type stability for the travel time difference data on simple manifolds (Theorem~\ref{thm:stability_of_TTDD}).

\subsection{Inverse problem 3: Broken scattering relations}

	Our third main result is the uniqueness result for the \emph{broken scattering relations} on simple manifolds.
	Let~$TM$ and~$SM$ be the tangent and the unit sphere bundles of $(M,g)$, respectively.
	The geodesic flow of $(M,g)$ is denoted by
	\[
	\phi\colon SM \times \R \to SM, \quad \phi_{t}(\eta)=(\gamma_\eta(t),\dot{\gamma}_{\eta}(t))
	,
	\]
    and $\gamma_{\eta}$ is the geodesic of $(M,g)$ given by the initial conditions
	$
	(\gamma_\eta(0),\dot{\gamma}_{\eta}(0))=\eta \in SM.
	$
	
	Let $\pi \colon TM \to M$ be the canonical projection to the base points given by $\pi(z,v)=\pi(v)=z$, and let $\nu$ be the inward pointing unit normal vector field on the boundary of~$M$. The boundary of the sphere bundle is the set
	\[
	\p SM= \{v\in SM: \pi(v)\in \p M\}.
	\]
	We call the set
	\[
	\inwb{} = \{v\in \p SM: \langle v ,\nu \rangle_g >0\}
	\]
	the inward pointing bundle at the boundary. Similarly, we define the outward pointing bundle as
	\[
	\outwb{} = \{v\in \p SM: \langle v ,\nu \rangle_g <0\}.
	\]
	The unit ball bundle of $\partial M$ is then defined by 
	\begin{equation}
		B\partial M
		=
		\{
		\xi\in T\partial M;
		\abs{\xi}<1
		\}.
	\end{equation}
	There are natural diffeomorphisms
    \[
    \begin{split}
    N_{\mathrm{in}}&\colon\inwb{} \to B\p M \quad \text{and}\\ N_{\mathrm{out}}&\colon\outwb{} \to B\p M
    \end{split}
    \]
    given by the fiberwise projection $N_{\mathrm{in/out}}(\xi)=\xi-\langle \xi, \nu\rangle_g \nu$ in the direction of the normal vector~$\nu$.
	The point of these maps is to identify the bundle $B\partial M$, which is intrinsic to $\partial M$, with the bundles $\inwb{}$ and $\outwb{}$ which in principle carry information of how the boundary sits within the unknown manifold~$M$.
	
	We are ready to define the broken scattering relations.
	\begin{defn}
		\label{def:BSR}
		For each $T>0$, the \emph{broken scattering relation} $\cB$ on $B\p M$ is such that the vectors $v_1,v_2 \in B\pM$ satisfy $v_1 \cB v_2$ if and only if there exist two numbers $t_1,t_2>0$ for which $t_1+t_2=T$ and $\pi(\phi_{t_1}(N_{\mathrm{in}}^{-1}v_1))=\pi(\phi_{t_2}(N_{\mathrm{in}}^{-1}v_2))$.
		
		The family $\{\cB: \: T>0\}$ of relations is called the broken scattering relations of Riemannian manifold $(M,g)$. 
	\end{defn}
	
	\begin{rem}
	    We emphasize that a study of broken scattering relations requires the knowledge of the unit ball bundle, which is equivalent to knowing the first fundamental form of $\p M$.  
	\end{rem}
	
	Our third main result (Theorem~\ref{thm:BSR_to_geometry}) shows that the broken scattering relations determine any simple Riemannian manifold up to a Riemannian isometry. A stronger version of this result (for $\dim (M)\geq 3$) has been presented originally in~\cite{kurylev2010rigidity}, where the authors only assumed that $(M,g)$ is a smooth compact Riemannian manifold with smooth boundary. In~\cite{de2020foliated} this result was extended to cover reversible Finsler metrics that satisfy a convex foliation condition.
	
	Up to our best knowledge, Theorem~\ref{thm:BSR_to_geometry} of this paper is novel in the two dimensional case. As in~\cite{kurylev2010rigidity}, our proof is based on a reduction step (Proposition~\ref{prop:BSR_to_travel_times}) to the travel time data. Due to simplicity, we can provide a new proof for the reduction. This and our Proposition~\ref{prop:TTD_to_geom} yield a novel streamlined proof for the uniqueness of the broken scattering relations for simple metrics of all dimensions.

	
	
	
	
	
	\subsection{Outline of the paper}
	In Sections~\ref{sec:TTD}, \ref{sec:TTDD} and~\ref{Sec:BSR} of this paper, we consider the travel time data, travel time difference data, and broken scattering relations, respectively.
	In Section~\ref{sec:DITTD} we consider a diffeomorphism invariant travel time and travel time difference data and provide a stability result for them. 
	
\subsection*{Acknowledgements}

JI was supported by the Academy of Finland (projects 
\\
351665 and 351656). TS was supported by the National Science Foundation under grant DMS-2204997. We thank Jonne Nyberg and Antti Kykk\"anen for discussions.

	\section{Uniquenss and Stability of the Travel Time Data}
	\label{sec:TTD}

	
	We recall that every simple Riemannian manifold is diffeomorphic to the closed unit ball $\bD$ of $\R^n$. Thus, from here onwards we study simple metrics on $\bD$.

	\begin{prop}
	\label{prop:isometry}
		Let $g$ be a simple Riemannian metric on $\bD$. The travel time map $\cR$, given by formula \eqref{eq:map_R}, is a metric isometry.
	\end{prop}
	\begin{proof}
	By triangle inequality the map $\cR$ is easily seen to be $1$-Lipschitz. To prove the reverse estimate, let $x,y \in \bD$. Since $g$ is simple, there exists a unique distance minimizing geodesic $\gamma$ connecting the points $x$ and $y$ to some boundary point $z \in \bS$.
	
	Since $\gamma$ exits $\bD$ at $z$, and all three points $x,y,z$ are on the same distance minimizing curve $\gamma$, we have
	\[
	|r_x(z)-r_{y}(z)|=|d(x,z)-d(y,z)|=d(x,y).
	\]
    Thus we have verified the equality
	\[
	\|\cR(x)-\cR(y)\|_\infty=d(x,y) \quad \text{ for all } x,y \in \bD.
	\]
	The proof is complete.
    %
    %
	\end{proof}
	
    Since $\bD$ is compact, the previous proposition yields that $\cR(\bD)$ is a compact subset of the Banach space $(C(\bS), \|\cdot\|_\infty)$. In order to compare two such image sets, we next define a notion of distance between them. 
	
	\begin{defn}
	\label{defn:Haus_of_travel_time_data}
	We set \textit{the distance of travel time data} of two simple Riemannian metrics $g_1$ and $g_2$ on $\bD$ to be
	\[
	\hd{C(\bS)}(\cR_1(\bD),\cR_2(\bD))\geq 0,
	\]
	where $\hd{C(\bS)}$ is the Hausdorff distance of $C(\bS)$ and the travel time maps $\cR_i$ are as in formula~\eqref{eq:map_R}.
	
	Moreover, we say that the travel time data of the simple Riemannian metrics $g_1$ and $g_2$ on $\bD$ \textit{coincide} if 
	\[
	\hd{C(\bS)}(\cR_1(\bD),\cR_2(\bD))=0.
	\]
	\end{defn}
	
     If $\Phi\colon \bD \to \bD$ is a diffeomorphism whose restriction on $\bS$ is the identity map and $g_1$ is any simple metric of $\bD$, then the pullback metric $g_2:=\Phi^\ast g_1$ is a simple metric on $\bD$ isometric to $g_1$. Thus the equality
	 \[
	 d_2(p,z)=d_1(\Phi(p),\Phi(z))=d_1(\Phi(p),z)
	 \]
	 (valid for all $p \in \bD$  and $z \in \bS$)
	 yields the equations $\cR_2(\bD)=\cR_1(\bD)$ and $\hd{C(\bS)}(\cR_1(\bD),\cR_2(\bD))=0$.
	 
	 In other words, the travel time data is invariant under boundary fixing diffeomorphisms. This is the natural gauge of the problem of determining the metric from its travel time data. In the following proposition we confirm that this is the only obstruction.

	\begin{prop}
	\label{prop:TTD_to_geom}
	If the travel time data of simple Riemannian metrics~$g_1$ and~$g_2$ on~$\bD$  coincide, then the map
	\[
	\cR_2^{-1}\circ \cR_1 \colon (\bD,g_1)\to (\bD,g_2),
	\]
	is a Riemannian isometry whose boundary restriction is the identity map.
	\end{prop}
	\begin{proof}
    Due to Definition~\ref{defn:Haus_of_travel_time_data} we have
    \begin{equation}
    \label{eq:data_coincide}
	 \cR_2(\bD)=\cR_1(\bD).
    \end{equation}
    Then by Proposition~\ref{prop:isometry}, the map 
	\[
	\Psi:=\cR_2^{-1}\circ \cR_1\colon (\bD,d_{1}) \to (\bD,d_{2})
	\]
	is a well defined bijective metric isometry. Thus by the Myers--Steenrod theorem, 
	the map $\Psi$ is a smooth Riemannian isometry.
	
	We get from equation \eqref{eq:data_coincide} that for every $z \in \bS$ we have
	\[
	\begin{split}
	d_2(\Psi(z),z)=d_1(z,z)=0,
	\end{split}
	\]
	which ensures that $\Psi$ is the identity at the boundary.
	\end{proof}

	To measure how close two compact metric spaces $X$ and $Y$ are to each other, we use the \emph{Gromov--Hausdorff distance}
	\begin{equation}
    \begin{split}
    \ghd{}(X,Y):=
    &
    \inf
    \{
    \hd{Z}(f(X),g(Y));
\\&\qquad
Z\text{ is a metric space},
\\&\qquad
f\colon X\to Z
\text{ and } g\colon Y\to Z
\\&\qquad
\text{ are isometric embeddings}
\}.
\end{split}
\end{equation}
	This distance is zero if and only if the metric spaces $X$ and $Y$ are isometric (see~\cite{gromov2007metric}). We are ready to formulate and prove our first main result:
	
	\begin{thm}[Stability of the Travel Time Data]
	\label{thm:stability}
	Let $n\geq 2$, and let $g_1$ and $g_2$ be two simple Riemannian metrics of $\bD$. Then
	\begin{equation}
		\label{eq:metric_in_eq}
		    \ghd{}((\bD,g_1),(\bD,g_2))
			\leq
			\hd{C(\bS)}(\cR_1(\bD),\cR_2(\bD)).
		\end{equation}
		In particular, if the travel time data for two metrics coincide, then they agree up to a boundary fixing isometry.
	\end{thm}
	\begin{proof}
	Inequality \eqref{eq:metric_in_eq} follows from Proposition~\ref{prop:isometry} and the definition of Gromov--Hausdorff distance.
	\end{proof}

	\section{Uniqueness and Stability of the Travel Time Difference Data}
	\label{sec:TTDD}

    \begin{prop}
    \label{prop:D_is_isometry}
    If $g$ is a simple metric on $\bD$, then its travel time difference map~$\cD$, given in Definition~\ref{def:travel_time_diff}, is a metric isometry.
    \end{prop}
    \begin{proof}
    First we observe that by triangle inequality the map $\cD$ is $1$-Lipschitz.
    For the reverse estimate, let $x,y\in \bD$. Since $g$ is a simple metric, there exists a unique globally distance minimizing geodesic $\gamma$ that goes through $x$ and $y$, having some endpoints $z,w \in \bS$. Since all four points $x,y,z,w$ are on the same distance minimizing curve, we have
    \[
    |(\cD(x)-\cD(y))(z,w)|=\frac{1}{2}|d(x,z)-d(y,z)+d(y,w)-d(x,w)|=d(x,y).
    \]
    Thus the map $\cD$ is a metric isometry as claimed.
    \end{proof}
	
	\begin{defn}
	\label{defn:Haus_of_travel_time_diff_data}
	We set \textit{the distance of the travel time difference data} of two simple Riemannian metrics $g_1$ and $g_2$ on $\bD$ to be
	\[
	\hd{C(\bS\times \bS)}(\cD_1(\bD),\cD_2(\bD))\geq 0,
	\]
	where $\hd{C(\bS\times \bS)}$ is the Hausdorff distance of the Banach space $(C(\bS\times \bS), \|\cdot\|_\infty)$ and the travel time difference maps $\cD_i$ are as in Definition~\ref{def:travel_time_diff}.
	
	Moreover, we say that the travel time difference data of the simple Riemannian metrics $g_1$ and $g_2$ on $\bD$ \textit{coincide} if 
	\[
	\hd{C(\bS)}(\cD_1(\bD),\cD_2(\bD))=0.
	\]
	\end{defn}
	
	Clearly the travel time difference data is invariant under boundary fixing diffeomorphisms. In the following proposition we show that this is the only obstruction of determining the metric from its travel time difference data.
	
	\begin{prop}
	\label{prop:TTDD_to_geom}
	If the travel time difference data of simple Riemannian metrics $g_1$ and $g_2$ on $\bD$  coincide, then the map
	\[
	\cD_2^{-1}\circ \cD_{1} \colon (\bD,g_1)\to (\bD,g_2),
	\]
	is a Riemannian isometry whose boundary restriction is the identity map.
	\end{prop}
	\begin{proof}
	By Definition~\ref{defn:Haus_of_travel_time_diff_data} we have
    \begin{equation}
    \label{eq:Ds_are_the_same}
	\cD_1(\bD)=\cD_2(\bD)
    \end{equation}
	as subsets of the Banach space $C(\bS \times \bS)$.	Therefore, Proposition~\ref{prop:D_is_isometry} implies that the map
	\[
	\Theta := \cD_2^{-1}\circ \cD_{1}\colon (\bD,d_1)\to (\bD,d_2)
	\]
	is a well defined bijective metric isometry. By the Myers--Steenrod theorem, the map $\Theta$ is a smooth Riemannian isometry.
	
	For every $z \in \bS$, equation~\eqref{eq:Ds_are_the_same} yields
	\[
	d_2(\Theta(z),z)=d_2(\Theta(z),z)-d_2(\Theta(z),\Theta(z))=d_1(z,z)-d_1(z,\Theta(z))=-d_1(\Theta(z),z),
	\]
	and so $\Theta$ is the identity on the boundary as claimed.
	\end{proof}
	
	
	Our second main result is as follows.
	
	\begin{thm}[Stability of the Travel Time Difference Data]
	\label{thm:stability_of_TTDD}
    Let $n\geq 2$, and let $g_1$ and $g_2$ be two simple Riemannian metrics of $\bD$. Then 
	\begin{equation}
		\label{eq:metric_in_eq_2}
		    \ghd{}((\bD,g_1),(\bD,g_2))
			\leq
			\hd{C(\bS\times \bS)}(\cD_1(\bD),\cD_2(\bD)).
		\end{equation}
		In particular, if the travel time difference data for two metrics coincide, then they agree up to a boundary fixing isometry.
	\end{thm}
	\begin{proof}
	Inequality \eqref{eq:metric_in_eq_2} follows from Proposition~\ref{prop:D_is_isometry} and the definition of Gromov--Hausdorff distance.
	\end{proof}

	\section{Uniqueness of the Broken Scattering Relations}
	\label{Sec:BSR}
	
	Suppose that $g_1$ is a simple metric on $\bD$ and $g_2$ is the pullback metric of $g_1$ under some boundary fixing diffeomorphism $\Phi$ of $\bD$. Then the first fundamental forms of these metrics on $\bS$ are the same. The broken scattering relations of these metrics also coincide. 

	In this section we show that two simple metrics of $\bD$ whose first fundamental forms on~$\bS$ and broken scattering relations coincide, agree up to a diffeomorphism fixing~$\bS$. This is obtained via the following reduction step to the travel time data.
	
	\begin{prop}
	\label{prop:BSR_to_travel_times}
	Let $g_1$ and $g_2$ be two simple Riemannian metrics on $\bD$ whose first fundamental forms agree on  $\bS$. If the broken scattering relations of these metric coincide, then their travel time data also agree.
	\end{prop}
	
	Proposition~\ref{prop:BSR_to_travel_times} will be proved at the end of this section with the help of auxiliary lemmas.
	We will first present and prove our third main result. 
	
	\begin{thm}[Uniqueness of the Broken Scattering Relations]
	\label{thm:BSR_to_geometry}
	Let $n\geq 2$, and let~$g_1$ and~$g_2$ be two simple Riemannian metrics of $\bD$ whose first fundamental forms on $\bS$ agree. If the broken scattering relations of these metric coincide, then there exists a smooth Riemannian isometry 
	$
	\Psi\colon (\bD,g_1)\to (\bD,g_2)
	$
	whose boundary restriction $\Psi\colon \bS \to \bS$ is the identity map.
	\end{thm}
	
	\begin{proof}
	By Proposition~\ref{prop:BSR_to_travel_times} the travel time data of the metrics $g_1$ and $g_2$ coincide, and we can choose the Riemannian isometry $\Psi$ as in Proposition~\ref{prop:TTD_to_geom}.
	\end{proof}
	
    The \emph{exit time function} $\exit\colon S\bD\to[0,\infty]$ of $(\bD,g)$ is defined so that the maximal domain of definition for the geodesic with initial conditions $\eta\in S\bD$ is $[-\exit(-\eta),\exit(\eta)]$. 
    If~$g$ is a simple metric,
    the exit time is always finite and continuous in $\eta \in S\bD$ (see e.g. \cite[Lemma 13]{de2020foliated}). Moreover, the boundary point $\gamma_{v}(\exit(\eta))$ is the first point where the geodesic $\gamma_v$ hits the boundary.

    \begin{lem}
	Let $g$ be a simple Riemannian metric of $\bD$.
		\label{lem:geo_intersect}
		Let $v_1, v_2 \in \inwd{}$ be such that 
		the geodesics $\gamma_{v_1}$ and $\gamma_{v_2}$ do not have exactly the same endpoints on the boundary. There exists $v_3 \in \inwd{}$ such that $\gamma_{v_3}$ intersects $\gamma_{v_1}$ but not $\gamma_{v_2}$.
	\end{lem}
	\begin{proof}
	Let $z\in\bS$ be and endpoint of $\gamma_{v_1}$ that is not an endpoint of $\gamma_{v_2}$.
	Therefore short geodesics with both endpoints sufficiently close to $z$ will not meet $\gamma_{v_2}$.
	(These geodesics can be thought of as being almost tangent to the boundary.)
	By strict convexity of the boundary at $z$ such geodesics exist, and some of them meet $\gamma_{v_1}$.
	We simply let $v_3$ to be the initial data of one of these geodesics.
	\end{proof}

	\begin{rem}
	Lemma~\ref{lem:geo_intersect} may fail if the boundary is not strictly convex.
	On the closed hemisphere of $\mathbb{S}^2$ every pair of geodesics intersects each other.
	This follows easily from geodesics being great circles, i.e., intersections of the sphere and a plane through the origin.
	The claim of the lemma is, however, true for hemispheres of higher dimension.
	\end{rem}
	\begin{lem}
	\label{lem:bsr_two_geo_equal}
	Let $g$ be a simple Riemannian metric on $\bD$. Let $v_1, v_2 \in \inwd{}$. The following two statements are equivalent:
	\begin{enumerate}
	\item We have $V(v_1)=V(v_2)$, where
\begin{equation}
    \label{eq:set_of_focusing_directions}
		V(v_i):=\{w\in B\bS: \text{there is } T>0 \text{ for which } N_{\mathrm{in}}(v_i)\cB w\}.
\end{equation}
	\item Either $v_1=v_2$ or $v_2=-\phi_{\exit(v_1)}(v_1)$.
	\end{enumerate}
	\end{lem}
	\begin{proof}
		Clearly~(2) implies~(1), and it suffices to prove the implication from~(1) to~(2).
		Suppose, towards contradiction, that condition~(2) is not true, so that the traces of geodesics $\gamma_{v_1}$ and $\gamma_{v_2}$ are not the same. Since $g$ is simple, these geodesics cannot have the same endpoints. Whence, by Lemma~\ref{lem:geo_intersect}, there exists $v_3 \in \inwd{}$ such that $\gamma_{v_3}$ intersects $ \gamma_{v_1}$ but does not intersect $\gamma_{v_2}$. This contradicts condition~(1). 
	\end{proof}
	
	The broken scattering relations \textit{a priori} record the information whether two 
	\\
	geodesics, starting at the boundary, intersect. Next we show that in the case of simple metrics, more geometric information can be easily 
	extracted from the broken scattering relations.
	
	For us the \emph{scattering relation} is the map $\inwd{}\ni \eta\mapsto-\phi_{\exit(\eta)}(\eta)\in\inwd{}$.
	The negative sign makes the relation map inward pointing vectors to inward pointing vectors.
%
%
	In the subsequent lemma, we show that the broken scattering relations determine the exit time function and the scattering relation on $\inwd{}$. 
	
	\begin{lem}
		\label{lem:bsr_det_sr}
		Let $g$ be a simple Riemannian metric of $\bD$.	The broken scattering relations determine the scattering relation and exit time function on $\inwd{}$.
	\end{lem}
	\begin{proof}
       First we observe that the broken scattering relations determine the exit time function via the equation
       	\[
		\exit(v)=\frac{1}{2}\sup\{T>0: (N_{\mathrm{in}}v) \cB (N_{\mathrm{in}}v)\}, \quad \text{ for any } v \in \inwd.
		\]
       Then we use these relations to find the set $V(v)$, as in equation \eqref{eq:set_of_focusing_directions}, for every $v \in \inwd{}$. Let $v_1,v_2 \in \inwd{}$. Then by Lemma~\ref{lem:bsr_two_geo_equal} we have 
		$
		V(v_1)=V(v_2)
		$
		if and only if $v_2 = v_1$ or $v_2=-\phi_{\exit(v_1)}(v_1)$. Thus, the broken scattering relations determine the scattering relation.
	\end{proof}
	
	\begin{rem}
	    After we have recovered the scattering relation and the exit time function on $\inwd{}$, together known as the \emph{lens relation}, we could reduce the proof of Theorem~\ref{thm:BSR_to_geometry} to the classical \emph{boundary rigidity problem}, which has been solved on simple 2-manifolds~\cite{pestov2005two} and for generic simple metrics in dimensions three and higher~\cite{stefanov2005boundary}. As the broken scattering relations carry more information than the lens relation, we can provide an easier proof for Theorem~\ref{thm:BSR_to_geometry} via Proposition~\ref{prop:BSR_to_travel_times}.
	\end{rem}
	
	\begin{lem}
	\label{lem:travel_time_bsr}
	Let $g$ be a simple Riemannian metric on $\bD$. Let $v_1, v_2 \in B\bS$. Suppose that there exist a number $T>0$ such that 
    $
	v_1\cB v_2.
	$
	Then the broken scattering relations determine the travel times $t_1,t_2>0$ that satisfy the equations
	\[
	T=t_1+t_2, \quad \text{ and } \quad \gamma_{N_{\mathrm{in}}^{-1}(v_1)}(t_1)=\gamma_{N_{\mathrm{in}}^{-1}(v_2)}(t_2).
	\]
	\end{lem}
	\begin{proof}
	Since $g$ is simple, each pair of distinct geodesics can intersect at most once, and for any two distinct $v,w\in B\bD$ the sets
	\[
	\mathcal{T}(v,w):=\{T>0; \: v\cB w\}
	\]
	are either empty or singleton sets. 
	Thus the function $T\colon B\bS \times B\bS \to \R \cup \{\infty\}$ defined by
    \[
	T(v,w)=
	\left\{
	\begin{array}{lll}
         \mathcal{T}(v,w), &  v\neq w, & \mathcal{T}(v,w)\neq \emptyset
	     \\
	     \infty, &  v\neq w, & \mathcal{T}(v,w)= \emptyset
	     \\
	     \exit(N^{-1}_{\mathrm{in}}(v)), & v=w, &
	\end{array}
	\right.
    \]
	can be found by using the broken scattering relations. In particular $T=T(v_1,v_2)$. 
	
     As the broken scattering relations determine the scattering relation, we know the directions $\eta_i=N_\mathrm{in}(-\phi_{\exit(N_\mathrm{in}^{-1}v_i)}(N_{\mathrm{in}}^{-1}v_i))\in B\bS$. Thus, there are some numbers $t_1,t_2,s_1,s_2\geq 0$ that satisfy the equations
    \[
    \begin{split}
    \gamma_{N_{\mathrm{in}}^{-1}(v_1)}(t_1)&=\gamma_{N_{\mathrm{in}}^{-1}(v_2)}(t_2)
    =
    \gamma_{N_{\mathrm{in}}^{-1}(\eta_1)}(s_1)=\gamma_{N_{\mathrm{in}}^{-1}(\eta_2)}(s_2),
    \\
    t_1+t_2&=T(v_1,v_2), \qquad  t_1+s_1=T(v_1,\eta_1),
    \\
    t_2+s_2&=T(v_2,\eta_2), \qquad t_1+s_2=T(v_1,\eta_2).
    \end{split}
    \]
    Since the linear system has a unique solution we obtain
    \[
    \begin{split}
    t_1&=\frac{1}{2}\left( T(v_1,v_2)-T(v_2,\eta_2)+T(v_1,\eta_2) \right) 
    \quad\text{and}
    \\
    t_2&=T(v_1, v_2)-t_1. 
    \end{split}
    \]
    This completes the proof.
	\end{proof}

    We are ready to prove Proposition~\ref{prop:BSR_to_travel_times}.

	\begin{proof}[Proof of Proposition~\ref{prop:BSR_to_travel_times}]
    First we assume that $g$ is a simple metric on $\bD$. Recall that $\nu$ is the inward pointing unit normal vector field on $\bS$ with respect to this metric. We define the set
	\[
	\mathcal{D}:=\{(s_0,z_0)\in \R \times \bS: \: s_0\in [0,\exit(\nu(z_0))]\},
	\]
and the map
\begin{equation}
\label{eq:map_E}
	E\colon \mathcal{D} \to \bD, \quad E(s_0,z_0)=\exp_{z_0}(s_0\nu(z_0)).
\end{equation}
	
	Let $(z_0,s_0)\in \mathcal D$. We define the set $V(\nu(z_0))$ as in equation \eqref{eq:set_of_focusing_directions}. 
	Since for each $v \in V(\nu(z_0))$ there exists a unique $T>0$ that satisfies the relation $N_{\mathrm{in}}(\nu(z_0))\cB v$, we can apply Lemma~\ref{lem:travel_time_bsr} to find two functions
	$
	t_{z_0},s_{z_0}\colon V(\nu(z_0))\to \R
	$
    that satisfy the equation
    $
    \gamma_{\nu(z_0)}(s_{z_0}(v))=\gamma_{v}(t_{z_0}(v)) \: \text{ for every } v \in V(\nu(z_0)).
    $
    Then we study the subset
    \[
    V(s_0,z_0)=\{v \in V(\nu(z_0)):\: s_{z_0}(v)=s_0\},
    \]
    and note that by simplicity this set is an embedded submanifold of $B\bS$, and the restriction of the canonical projection $\pi \colon B\bS \to \bS$ on $V(s_0,z_0)$ is a diffeomorphism onto $\bS$. Moreover, we have
\[
    t_{z_0}(\pi|_{V(s_0,z_0)}^{-1}(q))=d(E(s_0,z_0),q) \quad \text{ for every } q \in \bS.
\]
	
	Since $(s_0,z_0)\in \mathcal{D}$ was arbitrary and the map $E$ is onto, we have
	\[
	\{t_{z_0}\circ \pi|^{-1}_{V(s_0,z_0)} \colon \bS \to \R: \: (z_0,s_0)\in \mathcal{D}\}=\cR(\bD).
	\]

	Let the metrics $g_1$ and $g_2$ be as in the assumption of this theorem. Since the first fundamental forms of these metric agree on $\bS$, they have the same unit normal vector fields on $\bS$. Since the broken scattering relations of these metrics also coincide, all the objects constructed in this and the preceding proofs are the same for both of these metrics. Therefore the travel time data of these metrics coincide. The proof is complete.
	\end{proof}
	
	\section{Diffeomorphism Invariant Data}
	\label{sec:DITTD}

	If $\Phi\colon \bD \to \bD$ is a diffeomorphism that does not need to be an identity on the boundary $\bS$, and $g_1$ is a simple metric on $\bD$, then the simple pullback metric $g_2:=\Phi^\ast g_1$ satisfies the equation
    \begin{equation*}
    \label{eq:pullback_dist}
	d_2(p,q)=d_1(\Phi(p),\Phi(q)) 
	\end{equation*}
	for all $p,q \in \bD$.
    However, the equation $\cR_1(\bD)=\cR_2(\bD)$ does not need to be true and in general the travel time data of Definition~\ref{defn:TTD} is not fully diffeomorphism invariant. 
    
    For instance, set $z=(1,0)$ and let $\Phi$ be a diffeomorphism of $\mathbb D^2$ for which $\Psi(-z)\neq -\Psi(z)$. Let $g_1$ be the Euclidean metric and $g_2=\Psi^*g_1$. There are many such $q \in \mathbb D^2 $ that the travel time function $d_1(q,\cdot)|_{\mathbb{S}}$ has its extreme values at $\pm z$, but only one $p\in \mathbb D^2$ (namely, $p=\Phi^{-1}(0)$), for which $\pm z$ are both extreme points of $d_2(p,\cdot)|_{\mathbb{S}}$. Therefore $\cR_1(\bD)\neq\cR_2(\bD)$.

	If we set 
    a diffeomorphism $\psi \colon \bS \to \bS$ to be the restriction of $\Phi^{-1}$ on $\bS$, we obtain 
    \[
    \cR_1(\bD)=\{d_2(q,\psi(\cdot))\colon \bS \to \R | \: q \in \bD\}=:\cR_{2,\psi}(\bD)
    \]
    and
    \[
    \cD_1(\bD)=\{d_2(q,\psi(\cdot))-d_2(q,\psi(\cdot))\colon \bS\times \bS \to \R | \: q \in \bD\}=:\cD_{2,\psi}(\bD).
    \]
    
    These suggest the following diffeomorphism invariant definitions for the difference of the travel time and of the travel time difference data.
    
    \begin{defn}
	Let $g_1$ and $g_2$ be two simple Riemannian metrics on $\bD$. We set the distance of their travel time data to be the number
	\[
	\cR(g_1,g_2):=\inf_{\psi}\{\hd{C(\bS)}(\cR_1(\bD),\cR_{2,\psi}(\bD)) | \:\: \psi \text{ is a diffeomorphism of } \bS\}.
	\]
	The distance of the travel time difference data of the metrics $g_1$ and $g_2$ is
	\[
	\cD(g_1,g_2):=\inf_{\psi}\{\hd{C(\bS)}(\cD_1(\bD),\cD_{2,\psi}(\bD)) | \:\: \psi \text{ is a diffeomorphism of } \bS\}.
	\]
	\end{defn}
	
    The last theorem of this paper is a diffeomorphism invariant version of Theorems~\ref{thm:stability} and~\ref{thm:stability_of_TTDD}.
    
    \begin{thm}[Stability of the Diffeomorphism Invariant Data]
	\label{thm:stability_2}
	Let $n\geq 2$, and let~$g_1$ and~$g_2$ be two simple Riemannian metrics of~$\bD$. Then
	\begin{equation}
		\label{eq:metric_in_eq_3}
		\begin{split}
		\ghd{}((\bD,g_1),(\bD,g_2)) &\leq \cR(g_1,g_2) \quad \text{and}\\
		\ghd{}((\bD,g_1),(\bD,g_2)) &\leq \cD(g_1,g_2).
		\end{split}
		\end{equation}
	In particular, $\cR(g_1,g_2)=0$ or $\cD(g_1,g_2)=0$ if and only if the metrics $g_1$ and $g_2$ are isometric.
	\end{thm}
	\begin{proof}
	We only provide the proof for the travel time data since the proof for the travel time difference data is analogous. For every $\eps>0$ there exists a diffeomorphism $\psi\colon \bS \to \bS$ satisfying the inequality
    \[
    \hd{C(\bS)}(\cR_1(\bD),\cR_{2,\psi}(\bD))<\cR(g_1,g_2)+\eps.
    \]
    If we define a map $\Psi\colon C(\bS) \to C(\bS)$ by the formula
    $
    \Psi(f)=f\circ \psi,
    $
    then Proposition~\ref{prop:isometry} yields that the maps 
    \[
    \begin{split}
    \cR_1&\colon (\bD,d_1)\to (C(\bS),\|\cdot\|_\infty)
    \quad \text{and}
    \\
    \Psi \circ \cR_2 &\colon (\bD,d_2)\to (C(\bS),\|\cdot\|_\infty) 
    \end{split}
    \]
    are metric isometries whose image sets are $\cR_1(\bD)$ and $\cR_{2,\psi}(\bD)$, respectively. Since $\eps>0$ can be chosen arbitrarily, the definition of Gromov--Hausdorff distance gives the inequality \eqref{eq:metric_in_eq_3}.
    
    If $\cR(g_1,g_2)=0$, then by the inequality \eqref{eq:metric_in_eq_3} the metric spaces $(\bD,d_1)$ and $(\bD,d_2)$ are isometric. Whence, the last claim follows from the Myers--Steenrod theorem.
	\end{proof}
	
	\bibliographystyle{abbrv}
    \bibliography{bibliography}
    
    \end{document}